\documentclass[12pt]{amsart}
\usepackage{amssymb, amsmath}
\usepackage[round]{natbib}
\usepackage{fullpage}

\usepackage{hyperref} 
\newtheorem{thm}{Theorem}[section]
\newtheorem{conj}{Conjecture}[section]
\newtheorem{prop}{Proposition}[section]
\newtheorem{lem}{Lemma}[section]
\newtheorem*{cor}{Corollary}

\theoremstyle{definition}

\theoremstyle{remark}
\newtheorem*{rem}{Remark}

\DeclareMathOperator{\des}{des}
\DeclareMathOperator{\cdes}{cdes}
\DeclareMathOperator{\ides}{ides}
\DeclareMathOperator{\asc}{asc}
\DeclareMathOperator{\iasc}{iasc}

\newcommand{\Sym}{\mathfrak{S}}
\newcommand{\eulerian}[2]{\left\langle\genfrac{}{}{0pt}{}{#1}{#2}\right\rangle}

\author{Mirk\'o Visontai}
\title[On the joint distribution of descents
and inverse descents]{Some remarks on the joint distribution of descents
and inverse descents}
\address{Department of Mathematics, Royal Institute of Technology,
SE-100 44 Stockholm, Sweden}
\email{visontai@kth.se}
\dedicatory{In memoriam Herb Wilf}

\begin{document}
\begin{abstract}
  We study the joint distribution of descents and inverse descents over
  the set of permutations of $n$ letters. Gessel conjectured that the two-variable
  generating function of this distribution can be expanded in a given basis with nonnegative
  integer coefficients. We investigate the action of the Eulerian operators that give the 
  recurrence for these generating functions. As a result we devise a 
  recurrence for the coefficients in question but are unable to settle the conjecture.
  
  We examine generalizations of the conjecture and obtain a type $B$ analog of the 
  recurrence satisfied by the two-variable generating function. 
  We also exhibit some connections to cyclic descents and cyclic inverse descents.
  Finally, we 
  propose a combinatorial model for the joint distribution 
  %of descents and inverse descents 
  in terms of statistics on inversion sequences. 
\end{abstract}

\maketitle

\section{Introduction}
Let $\mathfrak{S}_n$ denote the set of permutations of $\{1, \dotsc, n\}$. The 
number of \emph{descents} in a permutation $\pi = \pi_1\dotso\pi_n$ is defined as 
$\des(\pi) = |\{ i : \pi_i > \pi_{i+1}\}|.$ 
Our object of study is the two-variable generating function of descents and \emph{inverse descents}:

\[ A_n(s,t) =  \sum_{\pi \in \Sym_n} s^{\des(\pi^{-1})+1} t^{\des(\pi)+1}\,. \]

The specialization of this polynomial to a single variable reduces to the \emph{classical 
Eulerian polynomial}:
\[A_n(t) = A_n(1,t) = \sum_{\pi \in \Sym_n} t^{\des(\pi)+1} = 
\sum_{k=1}^n \eulerian{n}{k} t^k\,.\]

Eulerian polynomials and their coefficients
play an important role (not only) in enumerative combinatorics. The classical, univariate polynomials 
are quiet well-studied---see, for example, \cite{Car59,FS70}. 
This cannot be said for the bivariate generating function for the 
pair of statistics $(\des, \ides)$. %; although some formulas for the 
%joint distribution were studied in \cite{FH05} and the references therein. 
Here and throughout this note
we will use the shorthand 
$\ides(\pi) = \des(\pi^{-1}).$ 
Our main motivation to study these polynomials is the following conjecture of Gessel 
which appeared in a recent article by \cite{Bra08}; see also a nice exposition by \cite{Pet12}.

\begin{conj}[Gessel] 
\label{conj:gessel}
For all $n \ge 1$, 
\[A_n(s,t) = \sum_{i,j} \gamma_{n,i,j} (st)^i (s+t)^j  (1+st)^{n+1-j-2i}\,,\]
where $\gamma_{n,i,j}$ are nonnegative integers for all $i,j \in \mathbb{N}.$
\end{conj}

If true, this decomposition would refine the following classical
result, the {$\gamma$-\emph{nonnegativity}} for the Eulerian polynomials $A_n(t)$.

\begin{thm}[Th\'eor\`eme 5.6 of \citet{FS70}]
\label{thm:FS}
\[A_n(t) = \sum_{i=1}^{\lceil n/2 \rceil} \gamma_{n,i} 
t^i (1+t)^{n+1-2i}\,,\]
where $\gamma_{n,i}$ are nonnegative integers  
for all $i \in \mathbb{N}.$
\end{thm}

Before giving their proof, let us 
recall the recurrence satisfied by the Eulerian polynomials:
\begin{equation}
A_n (t) = n t A_{n-1}(t) + t(1-t) \frac{\partial}{\partial t} A_{n-1}(t)\,, 
\mbox{ for } n\ge 2\, ,
\label{eq:eulerian}
\end{equation}
with initial value $A_1(t) = t$.

\citet[Chapitre V]{FS70} give a purely algebraic proof of Theorem~\ref{thm:FS} 
by considering the \emph{homogenized} Eulerian polynomial, of degree $n+1$,
\begin{equation}\begin{split}
 A_n(t;y) &=  y^{n+1}A_n(t/y) \\
 &= \sum_{\pi \in \Sym_n}  t^{\des(\pi)+1}y^{\asc(\pi)+1}\,, 
\label{eq:hompolydefn}
\end{split}
\end{equation}
where $\asc(\pi)$ denotes the number of ascents ($\pi_i<\pi_{i+1}$)
in the permutation $\pi = \pi_1 \dotso \pi_n$. Note that this polynomial is different
from and therefore should not be confused with 
$A_n(s,t)$. To avoid confusion we use a semicolon and different variables.
We include their proof next, as we will be applying the same idea to the 
joint generating polynomial of descents and inverse descents in 
Section~\ref{sec:recurrencegamma}.

\begin{proof}[Proof of Theorem~\ref{thm:FS}]
The homogenized Eulerian polynomials defined in (\ref{eq:hompolydefn}) 
satisfy the recurrence
\begin{equation}
\label{eq:homogenizedEulerianRec}
A_{n}(t;y) = ty\left(\frac{\partial}{\partial t} A_{n-1}(t;y)+ 
\frac{\partial}{\partial y}A_{n-1}(t;y)\right), \mbox{ for } n\ge 2\,,
\end{equation}
which follows from observing the effect on the number of descents and ascents
of inserting the letter $n$ into a permutation 
of $\{1, \dotsc, n-1\}$. Compare this with the recurrence in (\ref{eq:eulerian}).

It is clear from symmetry observations that $A_n(t;y)$ can be written (uniquely) 
in the basis 
\[\left\{(ty)^i(t+y)^{n+1-2i}\right\}_{i=1\dotso \lceil\frac{n}{2}\rceil}\]
with some coefficients $\gamma_{n,i}$. To show that $\gamma_{n,i}$ are in fact
nonnegative integers consider the action of the operator 
$T = ty\left({\partial}/{\partial t} + {\partial}/{\partial y}\right)$ on 
a basis element. 
Apply $T$ on the $i$th basis element we get that
\[ T[(ty)^i(t+y)^{n+1-2i}] = i(ty)^i(t+y)^{n+2-2i} + 
2(n+1-2i)(ty)^{i+1}(t+y)^{n-2i}, \]
which in turn implies the following recurrence on the coefficients:
\begin{equation}
\gamma_{n+1,i} = i \gamma_{n,i} + 2(n+3-2i)\gamma_{n,i-1}.
\label{rec:gamma_ni}
\end{equation}

The statement of Theorem~$\ref{thm:FS}$ now follows, since the initial values are
nonnegative integers, in particular, $\gamma_{1,1}= 1$ and $\gamma_{1,i} = 0$ 
for $i\ne 1$. Furthermore, the constraint $1\le i \le \lceil\frac{n}{2}\rceil$ assures that 
both positivity and integrality are preserved by recurrence~(\ref{rec:gamma_ni}). 

\end{proof}

\begin{rem}
The study of these so-called Eulerian operators goes back to Carlitz as was 
pointed out to the author by I.~Gessel. See, for example, \cite{Car73} for a slightly different 
variant of $T$. Also, the operator $t\left(n+ (1-t) (\partial/\partial t)\right)$ is closely related to 
a special case of a generalized derivative operator already studied by Laguerre, called
\textit{\'emanant} or polar derivative; see, for example, 6.~in \cite{Mar35}. 
\end{rem}

Finally, we must also mention the ``valley-hopping'' proof of Theorem~\ref{thm:FS} 
by \citet*[Proposition 4]{SWG83} 
which is a beautiful construction that proves that the coefficients $\gamma_{n,i}$ are not only
nonnegative integers but that they are, in fact, cardinalities of certain equivalence 
classes of permutations. Their proof is part of a more general phenomenon, an action of
transformation groups on the symmetric group $\Sym_n$ studied by
\citet{FS74}.

\section{Symmetries of $A_n(s,t)$ and a homogeneous recurrence}

The polynomials $A_n(s,t)$ were first studied by \cite*{CRS66}.
They proved a recurrence for the coefficients of $A_n(s,t)$ (see equation (7.8) in their article---note
there is an obvious typo in the last row of the equation, cf.~equation (7.7) in the same article). The
recurrence they provide for the coefficients is equivalent to the following one for 
the generating functions. 
\begin{thm}[Equation (9) of \citet*{Pet12}] For $n\ge 2$, 
\[
\begin{split}
nA_n(s,t) = &\left(n^2st+(n-1)(1-s)(1-t)\right)A_{n-1}(s,t)\\
 &+ nst(1-s)\frac{\partial}{\partial s}A_{n-1}(s,t) + 
nst(1-t)\frac{\partial}{\partial t}A_{n-1}(s,t)\\ 
&+st(1-s)(1-t)\frac{\partial^2}{\partial s\partial t}A_{n-1}(s,t)\,,
\end{split}
\]
with initial value $A_1(s,t) = st.$
\label{thm:Carlitz}
\end{thm}

At first glance, this recurrence might not seem very useful at all. However, if we introduce additional 
variables---to count ascents ($\asc$) and inverse ascents ($\iasc$)---we obtain a more transparent
recurrence. 
So, let us first define 
\begin{align} 
A_n(s,t;x,y) &=  \sum_{\pi \in \Sym_n} 
s^{\ides(\pi)+1} t^{\des(\pi)+1}x^{\iasc(\pi)+1} y^{\asc(\pi)+1}\\
&=\sum_{\pi \in \Sym_n} 
s^{\ides(\pi)+1} t^{\des(\pi)+1}x^{n-\ides(\pi)} y^{n-\des(\pi)}\\
&= (xy)^{n+1} A_n(s/x, t/y)\,.\label{eq:homogenized4var}
\end{align}

\begin{prop}
$A_n(s,t;x,y)$ is homogeneous of degree $2n+2$ and is invariant under the 
action of the Klein 4-group $V \cong \langle \mathrm{id}, (12)(34), (13)(24), (14)(23)\rangle$,
where the action of $\sigma \in V$ on $A_n(s,t;x,y)$ is permutation of the variables
accordingly (e.g., $\sigma=(13)(24)$ swaps $x$ with $s$ and $y$ with $t$, simultaneously).
\end{prop}
\begin{proof}
The homogeneity is immediate from the second line of the equation above. The invariance
is a consequence of 
the symmetry properties of 
$A_n(s,t)$, such as $A_n(s,t) = A_n(t,s)$; see, for example, equations (12--14) 
in \citep{Pet12}.
Note that, due to the introduction of the new variables, for $n \ge 4$, the 
polynomial $A_n(s,t;x,y)$ is \emph{not} symmetric. 
\end{proof}

Now we are in position to give our homogeneous recurrence.
\begin{thm} For $n\ge 2$, 
\begin{equation}
\begin{split}
nA_n(s,t;x,y) = &(n-1)(s-x)(t-y)A_{n-1}(s,t;x,y)\\
 &+ stxy\left(\frac{\partial}{\partial s} + 
 \frac{\partial}{\partial x}\right)\left(\frac{\partial}{\partial t}+
 \frac{\partial}{\partial y}\right)A_{n-1}(s,t;x,y)
\end{split}
\label{eq:fourvariable}
\end{equation}
with initial value $A_1(s,t;x,y) = stxy.$
\label{thm:fourvariable}
\end{thm}
\begin{proof}

Consider the bivariate recurrence given in Theorem \ref{thm:Carlitz} and
observe that it can be rewritten as 
\[
nA_n(s,t) = \left((n-1)(1-s)(1-t)
 + st\left(n+(1-s)\frac{\partial}{\partial s}\right)
 \left(n + (1-t)\frac{\partial}{\partial t}\right)\right)
 A_{n-1}(s,t).
\]
Now we can make both sides of the equation homogeneous using (\ref{eq:homogenized4var}).
Since the two Eulerian operators act on different variables each of them can be 
replaced by their symmetric two-variable homogenized counterpart and the theorem follows.
\end{proof}

\begin{rem}
The invariance of $A_n(s,t;x,y)$ under the Klein-group action also follows easily 
from recurrence (\ref{eq:fourvariable}) directly. Clearly, $A_1(s,t;x,y) = stxy$ 
is invariant under the action of the group (in fact, it is symmetric) and the 
operator acting on $A_{n}(s,t;x,y)$ denoted by
\begin{equation}
T_{n} = n(s-x)(t-y) + stxy\left(\frac{\partial}{\partial s} + 
 \frac{\partial}{\partial x}\right)\left(\frac{\partial}{\partial t}+
 \frac{\partial}{\partial y}\right). 
\label{eq:Toperator}
\end{equation} itself is invariant under the action of the Klein-group. 
\end{rem}

Finally, Theorem~\ref{thm:fourvariable} allows us to give a (homogenized) restatement of 
Gessel's conjecture:
\begin{conj} 
\[A_n(s,t;x,y) = \sum_{i,j} \gamma_{n,i,j} (stxy)^i(st+xy)^j
(tx+sy)^{n+1-2i-j} ,\]
where $\gamma_{n,i,j} \in \mathbb{N}$ for all $i,j \in \mathbb{N}.$
\end{conj}

For example, we have (cf. page 18 of \cite{Pet12}):
\begin{align*}
A_1(s,t;x,y) &= stxy \\
A_2(s,t;x,y) &= stxy(st+xy) \\
A_3(s,t;x,y) &= stxy(st+xy)^2 +2(stxy)^2\\
A_4(s,t;x,y) &= stxy(st+xy)^3 + 7(stxy)^2(st+xy) + (stxy)^2(tx+sy)\\
A_5(s,t;x,y) &= stxy(st+xy)^4+ 16(stxy)^2(st+xy)^2 +6(stxy)^2(st+xy)(tx+sy) +16(stxy)^3.
\end{align*}

\begin{rem}
It is not too hard to see that Theorem~\ref{thm:fourvariable} is, in fact, equivalent 
to Theorem~\ref{thm:Carlitz}. At the same time, the symmetric nature of the 
homogeneous operator is more suggestive to combinatorial interpretation. 
It would be nice to find such an interpretation (perhaps in terms of non-attacking 
rook placements on a rectangular board). 
\end{rem}

\medskip

\section{A recurrence for the coefficients $\gamma_{n,i,j}$}
\label{sec:recurrencegamma}

Following the ideas of \citet[Chapitre V]{FS70} that were used to devise a recurrence
for $\gamma_{n,i}$, we apply the operator $T_n$ to the basis elements to obtain a 
recurrence for the coefficients $\gamma_{n,i,j}$. As a result, we obtain the following recurrence.

\begin{thm} Let $n\ge 1$. For all $i\ge1$ and $j \ge 0$, we have
\begin{equation}
\begin{split}
(n+1) \gamma_{n+1,i,j} =&\quad (n+i(n+2-i-j))\gamma_{n,i,j-1} 
+ (i(i+j)-n) \gamma_{n,i,j} \\&+ 
(n+4-2i-j)(n+3-2i-j) \gamma_{n, i-1,j-1}+\\&+
(n+2i+j)(n+3-2i-j)\gamma_{n, i-1,j}  \\&+
(j+1)(2n+2-j)\gamma_{n,i-1,j+1} + (j+1)(j+2)\gamma_{n,i-1,j+2}\, ,
\end{split}
\label{rec:gamma_nij}
\end{equation}
with $\gamma_{1,1,0} = 1$, $\gamma_{1,i,j} = 0$ (unless $i=1$ and $j=0$) and 
$\gamma_{n,i,j} = 0$ if $i < 1$ or $j < 0$.
\end{thm}
\begin{proof}
Denote the basis elements by $B^{(n)}_{i,j} = (stxy)^i(st+xy)^j(tx+sy)^{n+1-2i-j}$ for convenience, and recall the definition of $T_n$ given in (\ref{eq:Toperator}).

A quick calculation shows that
\begin{equation}
n(s-x)(t-y)B^{(n)}_{i,j} = n \left(B^{(n+1)}_{i,j+1}- B^{(n+1)}_{i,j}\right) \,.
\label{eq:M}
\end{equation}

To calculate the action of the differential operators on the basis elements, we use the product rule,
which for second-order partial derivatives is given by the following formula:
\[
\begin{split}
\partial_{zw}(fgh) =& \quad \partial_{zw}(f)gh+\partial_{z}(f)\partial_{w}(g)h+\partial_{z}(f)g\partial_{w}(h) \\
&+\partial_{w}(f)\partial_z(g)h+f\partial_{zw}(g)h+f\partial_z(g)\partial_{w}(h) \\
&+\partial_w(f)g\partial_{z}(h)+f\partial_{w}(g)\partial_z(h)+fg\partial_{zw}(h)\, ,
\end{split}
\] where $f,g,h$ are functions, $\partial_z = \partial/\partial z$ and $\partial_w = \partial/\partial w$
denote the partial differential operators with respect to $z$ and $w$, and $\partial_{zw} = \partial_z\partial_w$ 
is the second-order differential operator.

After some calculations, this gives the following:
\begin{equation}
\begin{split}
stxy\left(\frac{\partial^2}{\partial s \partial t} + \frac{\partial^2}{\partial x \partial y}\right)
B^{(n)}_{i,j} &= i(n+1-i-j)B^{(n+1)}_{i,j+1}+ j(2n+3-j) B^{(n+1)}_{i+1,j-1}\\
&\quad+ (n+1-2i-j)(n-2i-j)B^{(n+1)}_{i+1,j+1}\,.
\end{split}
\label{eq:D1}
\end{equation}

\begin{equation}
\begin{split}
stxy\left(\frac{\partial^2}{\partial s \partial y} + \frac{\partial^2}{\partial t \partial x}\right)B^{(n)}_{i,j} &= i(i+j)B^{(n+1)}_{i,j} + j(j-1) B^{(n+1)}_{i+1,j-2}+\\&\quad
(n+1-2i-j)(n+2+2i+j)B^{(n+1)}_{i+1,j}\,.
\end{split}
\label{eq:D2}
\end{equation}

Summing (\ref{eq:M}), (\ref{eq:D1}) and 
(\ref{eq:D2}) we arrive at the following expression.
\[ 
\begin{split}
T_n[B^{(n)}_{i,j}] = &(n+i(n+1-i-j)) B^{(n+1)}_{i,j+1} + 
(i(i+j)-n) B^{(n+1)}_{i,j} \\ 
&+ (n+1-2i-j)(n-2i-j)B^{(n+1)}_{i+1,j+1} + 
(n+2+2i+j)(n+1-2i-j)B^{(n+1)}_{i+1,j} \\
&+ j(2n+3-j) B^{(n+1)}_{i+1,j-1}+ 
j(j-1) B^{(n+1)}_{i+1,j-2}.
\end{split}
\]

Finally, collecting together all terms $T_n[B^{(n)}_{k,\ell}]$ which 
contribute to $B^{(n+1)}_{i,j}$ we obtain (\ref{rec:gamma_nij}).

\end{proof}

\begin{rem} If we sum up both sides of (\ref{rec:gamma_nij}) for all
possible $j$ then we get (\ref{rec:gamma_ni}) back. 
\end{rem}

One could study the generating function 
\[G(u,v,w) = \sum_{i,j}\gamma_{n,i,j} u^n v^{i} w^{j}\] with coefficients satisfying
the above recurrence. Gessel's conjecture is equivalent to saying that its coefficients are 
nonnegative integers. Unfortunately, these properties are not immediate from the
recurrence (\ref{rec:gamma_nij}).

\section{Generalizations of the conjecture}

\cite{Ges12} noted that the following equality of \citet{CRS66}
\[ \sum_{i,j=0}^\infty \binom{ij+n-1}{n} s^it^j = 
\frac{A_n(s,t)}{(1-s)^{n+1}(1-t)^{n+1}}\]
can be generalized as follows. 

Let $\tau \in \mathfrak{S}_n$ with $\des(\tau) = k-1.$
Define $A_n^{(k)}(t)$ by
\[ \sum_{i,j=0}^\infty \binom{ij+n-k}{n} s^it^j = 
\frac{A^{(k)}_n(s,t)}{(1-s)^{n+1}(1-t)^{n+1}}.\]
Then the coefficient of $s^it^j$ in $A_n^{(k)}$ is the number of pairs of
permutations $(\pi,\sigma)$ such that $\pi\sigma = \tau$, $\des(\pi) = i$
and $\des(\sigma) = j$. \cite{Ges12} also pointed out that these polynomials
arise implicitly in \cite{MP70}; compare (11.10) there with the above equation.
%ADDREF Diaconis-Fulman.

This suggests that Conjecture~\ref{conj:gessel} holds in a more general form 
(this version of the conjecture appeared as Conjecture 10.2 in \cite{Bra08}).
\begin{conj}[Gessel]
\label{conj:gesseltau}
Let $\tau \in \mathfrak{S}_n$. Then 
\[\sum_{\pi \in \mathfrak{S}_n} s^{\des(\pi)+1}t^{\des(\pi^{-1}\tau)+1} = 
\sum_{i,j} \gamma^{\tau}_{n,i,j} (st)^i (s+t)^j  (1+st)^{n+1-j-2i}\,,\]
where $\gamma^{\tau}_{n,i,j}$ are nonnegative integers for all $i,j \in \mathbb{N}.$
Furthermore, the coefficients $\gamma^{\tau}_{n,i,j}$ do not depend on the actual
permutation $\tau$, only on the number of descents in $\tau$.
\end{conj}

In the special case when $\tau = n (n-1) \dotso 2 1$ (and hence $\des(\tau) = n-1$) the roles
of descents and ascents interchange.
\begin{thm} For $n\ge 2$, 
\begin{equation}
\begin{split}
nA^{(n)}_n(s,t;x,y) = &(n-1)(x-s)(t-y)A^{(n-1)}_{n-1}(s,t;x,y)\\
 &+ stxy\left(\frac{\partial}{\partial s} + 
 \frac{\partial}{\partial x}\right)\left(\frac{\partial}{\partial t}+
 \frac{\partial}{\partial y}\right)A^{(n-1)}_{n-1}(s,t;x,y)
\end{split}
\end{equation}
with initial value $A^{(1)}_1(s,t;x,y) = stxy$.
\end{thm}

In particular, we have the following identity.
\begin{cor}
\[A^{(n)}_n(s,t;x,y) = A_n(s,y;x,t).\] 
\end{cor}

\subsection{A type B analog}
\citet{Ges12} also noted that there is an analogous definition for the hyperoctahedral
group $\mathfrak{B}_n$. The elements of $\mathfrak{B}_n$ can be thought of as
signed permutations of $\{1, \dotsc, n\}$, and the type $B$ descents  are defined as
$\des_B(\sigma) = \{ i \in \{0,1, \dotsc, n\} : \sigma_i > \sigma_{i+1}\}$ with
$\sigma_0 := 0$ for 
$\sigma = \sigma_1\dotso\sigma_n \in \mathfrak{B}_n$.

\[ \sum_{i,j=0}^\infty \binom{2ij+i+j +1+n-k}{n} s^it^j = 
\frac{B^{(k)}_n(s,t)}{(1-s)^{n+1}(1-t)^{n+1}}\,,\]
where 
\[B^{(k)}_n(s,t) = \sum_{\sigma\in \mathfrak{B}_n} 
s^{\des_B(\sigma)}t^{\des_B(\sigma^{-1}\tau)},\]
with $\tau \in \mathfrak{B}_n$ such that $\des_B(\tau) = k-1$ (here 
$\des_B$ denotes the descents of type $B$).

Therefore, mimicking the proof of Theorem~\ref{thm:Carlitz} given by 
\citet{Pet12}, we get an analog of 
Theorem~\ref{thm:Carlitz} for the type $B$ %homogenous
two-sided Eulerian polynomials, $B_n(s,t) = B_n^{(1)}(s,t)$%, i.e., 
%the case when $k=1$..
%, $B_n(s,t;x,y) = (xy)^n B_n(s/x,t/y)$. This is the homogeneous
%analog of $B_n(s,t) = B_n^{(1)}(s,t)$, the case when $k=1$. 
\begin{thm} For $n\ge 2$, 
\begin{equation}
\begin{split}
nB_n(s,t) = &(2n^2st -nst + n)B_{n-1}(s,t)\\ &+  
(2nst(1-s) + s(1-s)(1-t))\frac{\partial}{\partial s}B_{n-1}(s,t)\\ &+
(2nst(1-t) + t(1-s)(1-t)) \frac{\partial}{\partial t}B_{n-1}(s,t)\\ &+
2st(1-s)(1-t)\frac{\partial^2}{\partial s\partial t}B_{n-1}(s,t)\,.
%nB_n(s,t;x,y) = &n(s-x)(t-y)B_{n-1}(s,t;x,y)\\
% &+ sx(y-t)\left(\frac{\partial}{\partial s} + 
% \frac{\partial}{\partial x}\right) B_{n-1}(s,t;x,y)\\
% &+ ty(x-s)\left(\frac{\partial}{\partial t}+
% \frac{\partial}{\partial y}\right)B_{n-1}(s,t;x,y)\\
% &+ 2stxy\left(\frac{\partial}{\partial s} + 
% \frac{\partial}{\partial x}\right)\left(\frac{\partial}{\partial t}+
% \frac{\partial}{\partial y}\right)B_{n-1}(s,t;x,y)
%\end{split}
%\end{equation}
%with initial value $B_1(s,t;x,y) = st+xy$.
\end{split}
\end{equation}
with initial value $B_1(s,t) = 1+st$.
\end{thm}

\begin{proof}
Following the proof for the case of the symmetric group in \cite[eq.~(9)]{Pet12}, 
we use the corresponding identity of binomial coefficients: 
\[n\binom{2ij+i+j +n}{n} = (2ij+i+j)\binom{2ij+i+j +n-1}{n-1} +
n\binom{2ij+i+j +n-1}{n-1}\,.\]
Multiplying both sides by the monomial $s^it^j$ and summing over all integers $i,j$ we get
\[\begin{split}
\sum_{i,j=0}^\infty &n\binom{2ij+i+j +n}{n} s^it^j = \\
&\sum_{i,j=0}^\infty (2ij+i+j)\binom{2ij+i+j +n-1}{n-1} s^it^j + 
\sum_{i,j=0}^\infty n\binom{2ij+i+j +n-1}{n-1} s^it^j
\,,
\end{split}
\]
%\[\begin{split}
%\sum_{i,j=0}^\infty &n\binom{2ij+i+j +n-1}{n} s^it^j = \\
%&\sum_{i,j=0}^\infty (2ij+i+j)\binom{2ij+i+j +n-2}{n-1} s^it^j + 
%\sum_{i,j=0}^\infty (n-1)\binom{2ij+i+j +n-2}{n-1} s^it^j
%\,,
%\end{split}
%\]
from which we obtain the following recurrence for 
$F_n(s,t) = B_n(s,t)/(1-s)^{n+1}(1-t)^{n+1}$: %Using the above identity 

\[nF_n(s,t) = 2st \frac{\partial^2}{\partial s\partial t} F_{n-1}(s,t) +
s \frac{\partial}{\partial s} F_{n-1}(s,t) +
t \frac{\partial}{\partial t} F_{n-1}(s,t) +
nF_{n-1}(s,t)\,.
\]
Now substitute back the expression for $F_n(s,t)$, multiply both sides with 
$(1-s)^{n+1}(1-t)^{n+1}$ and with a little work we get that
\[\begin{split}
nB_n(s,t) = &(2n^2st + nt(1-s) + ns(1-t) + n(1-s)(1-t))B_{n-1}(s,t)\\ &+  
(2nst(1-s) + s(1-s)(1-t))\frac{\partial}{\partial s}B_{n-1}(s,t)\\ &+
(2nst(1-t) + t(1-s)(1-t)) \frac{\partial}{\partial t}B_{n-1}(s,t)\\ &+
2st(1-s)(1-t)\frac{\partial^2}{\partial s\partial t}B_{n-1}(s,t)\,.
\end{split}
\]
%From this the theorem follows by the same reasoning as in the proof of 
%Theorem~\ref{thm:fourvariable}; note that $B_n(s,t;x,y) = (xy)^{n}
%B_n(s/x, t/y)$.
\end{proof}

It would be of interest to find a homogeneous version of this theorem 
(an analogue of Theorem~\ref{thm:fourvariable}) and a recurrence
for the corresponding $\gamma_{n,i,j}$ coefficients in the case of type $B$.

\subsection{Cyclic descents}
One can also consider two-sided Eulerian-like polynomials using cyclic descents. 
A \emph{cyclic descent} of a permutation $\pi$ in 
$\mathfrak{S}_n$ is defined as 
\[\cdes(\pi) =  |\{ i : \pi_i > \pi_{(i+1)\bmod n}\}|  = 
\des(\pi) + \chi(\pi_n > \pi_1)\,,\]
where 
\[
\chi(a>b) = \begin{cases}
1, $ if $ a>b $, and $\\
0, $ otherwise.$
\end{cases}\]

The following theorem refines a result of \citet[Corollary 1]{Ful00}.
\begin{thm}
\label{thm:cyclic}
For $n\ge 1$,
\[(n+1)A_n(s,t) = 
\sum_{\pi \in \mathfrak{S}_{n+1}} s^{\cdes(\pi^{-1})} t^{\cdes{\pi}}\,.\]
\end{thm}

\begin{lem}
\label{lem:cyclic}
Let $\sigma = 23\dotso n 1$ denote the cyclic rotation in $\mathfrak{S}_n$ (for 
$n \ge 2)$. Then
\[ (\cdes(\pi), \cdes(\pi^{-1})) = (\cdes(\pi\sigma), \cdes((\pi\sigma)^{-1})).\]
In other words, the cyclic rotation simultaneously preserves the cyclic descent and the
cyclic inverse descent stastics.
\end{lem}

\begin{rem} Lemma~\ref{lem:cyclic} is essentially the same as Theorem 6.5 in \citep{LP12}.
We give an elementary proof of it, for the sake of completeness.
\end{rem}

\begin{proof}
The part that $\cdes(\pi) = \cdes(\pi\sigma)$ is obvious since cyclical rotation does not 
effect the cyclic descent set. For the other part, it is equivalent to show that  
$\cdes(\pi) = \cdes(\sigma^{-1}\pi)$. In other words, the cyclic descent statistic is invariant under the operation when we cyclically shift the values of a permutation, i.e., add $1$ to each entry modulo $n$. For $\pi = \pi_1\dotso\pi_n$ an arbitrary permutation in 
$\mathfrak{S}_n$ denote the entry preceding $n$ and following $n$ by $a$ and $b$, 
respectively. Then $\pi = \pi_1\dotso a n b \dotso \pi_n$ and  
$\sigma^{-1}\pi = (\pi_1+1)\dotso (a+1)1(b+1) \dotso (\pi_n+1)$. Clearly, in all but one position the cyclic descents are preserved, same is true for the cyclic ascents. The $a\nearrow n$ cyclic ascent is replaced by the $(a+1) \searrow 1$ cyclic descent and similarly, 
$n \searrow b$ gets replaced by $1\nearrow (b+1)$. Thus, the total number of cyclic descents remains the same. 
\end{proof}

\begin{proof}[Proof of Theorem~\ref{thm:cyclic}]
Using Lemma~\ref{lem:cyclic} we can apply the cyclic rotation to any permutation in 
$\Sym_{n+1}$ until $\pi_{n+1} = n+1$. This will map exactly $n+1$ permutations 
in $\Sym_{n+1}$ to the same permutation $\pi_1\dots\pi_n(n+1)$. Clearly, 
$\cdes(\pi_1\dots\pi_n(n+1)) = \des(\pi_1\dots\pi_n)+1$ and 
$\cdes((\pi_1\dots\pi_n(n+1))^{-1}) = \des((\pi_1\dots\pi_n)^{-1})+1$ and the theorem 
follows.
\end{proof}

\section{Connection to inversion sequences}
\label{sec:invseq}

We conclude by proposing a combinatorial model for the joint distribution of descents 
and inverse descents. 

A permutation $\pi \in \Sym_n$ can be encoded as its inversion sequence 
$e = (e_1, \dotsc, e_n)$, where \[e_j = |\{i : i < j, \pi_i > \pi_j\}|.\] Let 
$I_n = \{(e_1, \dotsc, e_n) \in \mathbb{Z}^n : 0\le e_i \le i-1\}$
denote the set of inversion sequences for $\Sym_n$.

Recently, \cite{SS12} studied the \emph{ascent} statistic 
$\asc_I(e) = \left|\left\{i :  e_i < e_{i+1} \right\}\right|$ 
for inversion sequences (and
their generalizations)
and showed that this statistic is \emph{Eulerian}, i.e., it is equidistributed with the
descent statistic over permutations. We use the subscript $I$ to emphasize that this is a statistic
for inversion sequences which is different from the ascent statistic for permutations used
earlier in the paper.

\citet{MR01} also studied this representation of permutations under the name
``subexceedant functions''. They considered the statistic that counts that
distinct entries in $e \in I_n$, $\mathrm{dst}(\pi) = | \{e_i : 1\le i \le n\}|$. They gave
multiple proofs of the following observation (which they attributed to Dumont) that
this statistic is also Eulerian.

\begin{prop}[Dumont]
\[
A_n(x) = \sum_{e\in I_n} x^{\mathrm{dst}(e)}\,. 
\]
\label{prop:dumont}
\end{prop}

In fact, the joint distribution ($\asc_I,\mathrm{dst}$) over inversion sequences 
seems to agree
with the joint distribution $(\des, \ides)$ of descents and inverse descents over permutations. 
\begin{conj}
\[A_n(s,t) = \sum_{e\in I_n} s^{\mathrm{dst}(e)}t^{\asc_I(e)+1} \,. \] 
\end{conj}
This observation clearly deserves a bijective proof. Such a proof 
might shed light on a combinatorial 
proof of recurrence (\ref{eq:fourvariable}). Note that it is not even clear to begin with 
why the right-hand side should be a symmetric polynomial in variables $s$ and $t$.

\section*{Acknowledgments}
 I thank Ira Gessel for discussing his conjecture with me and for giving me feedback several 
 times during the preparation of this manuscript. His feedback, suggestions on notation and missing 
 references improved the presentation substantially. 
 I also thank Mireille Bousquet-M\'elou for sharing
 her observations on the recurrence of the $\gamma$ coefficients, and Carla Savage for 
 numerous discussions on inversion sequences and Eulerian polynomials that inspired 
 Section~\ref{sec:invseq}. I thank T.~Kyle Petersen for enlightening discussions, 
 Petter Br\"and\'en for helpful comments and for pointing
 out that Proposition~\ref{prop:dumont} was already known.
 Finally, I thank my advisor, Jim Haglund
 for his comments and guidance. A preliminary version of this article appeared in my dissertation written under his supervision.

\bibliographystyle{abbrvnat}
\bibliography{gessel}
\label{sec:biblio}

\end{document}